\documentclass[10pt,a4paper,reqno]{amsart}
\usepackage{amsthm}
\usepackage{amsmath}
\usepackage{amssymb}
\usepackage{graphicx,color}

\usepackage[font=small]{caption}
\usepackage{subcaption}

\usepackage{multirow}
\usepackage[shortlabels]{enumitem}
\usepackage{soul} 

\usepackage{cancel}

\makeatletter
\theoremstyle{plain}
\newtheorem{thm}{Theorem}
  \theoremstyle{definition}
  \newtheorem*{thm*}{Theorem}
  \newtheorem{defn}[thm]{Definition}
  \theoremstyle{remark}
  \newtheorem{rem}[thm]{Remark}
  \theoremstyle{plain}
  \newtheorem{prop}[thm]{Proposition}
  \theoremstyle{plain}
  \newtheorem{lem}[thm]{Lemma}
  \theoremstyle{plain}
  \newtheorem{cor}[thm]{Corollary}
 \theoremstyle{definition}
  
  \theoremstyle{remark}
  \newtheorem*{rem*}{Remark}

  \theoremstyle{definition}

\usepackage{amsfonts}
\usepackage{mathrsfs}

\newtheorem*{question*}{\it{QUESTION}}

\theoremstyle{plain}

\newcommand{\N}{\mathbb{N}}
\newcommand{\R}{{\mathbb{R}}}
\newcommand{\C}{{\mathbb{C}}}

\newcommand{\Z}{{\mathbb{Z}}}
\newcommand{\dd}{{\rm d}}
\newcommand{\ii}{{\rm i}}
\newcommand{\ee}{{\rm e}}

\newcommand{\p}{{\rm p}}

\newcommand{\ess}{{\rm ess}}

\newcommand{\spn}{\mathop\mathrm{span}\nolimits} 
\newcommand{\Dom}{\mathop\mathrm{Dom}\nolimits}

\makeatletter 
\@namedef{subjclassname@2020}{%
  \textup{2020} Mathematics Subject Classification}
\makeatother

\definecolor{DarkGreen}{rgb}{0,0.5,0.1} 

\makeatother

\begin{document}

\title[]{On the ground state of lattice Schrödinger operators}

\author{Michal Jex}
\address[Michal Jex]{
	Department of Physics, Faculty of Nuclear Sciences and Physical Engineering, Czech Technical University in Prague,  Břehová 7,115 19 Praha 1, Czech Republic
	}	
\email{michal.jex@fjfi.cvut.cz}

\author{Franti\v sek \v Stampach}
\address[Franti{\v s}ek {\v S}tampach]{
	Department of Mathematics, Faculty of Nuclear Sciences and Physical Engineering, Czech Technical University in Prague, Trojanova~13, 12000 Praha~2, Czech Republic
	}	
\email{stampfra@cvut.cz}



\date{\today}

\begin{abstract}
We prove necessary and sufficient conditions for lattice Schrödinger
operators to have a zero energy bound state in arbitrary dimension. The two criteria are sharp, complementary, and depend crucially on both the dimension and asymptotic behaviour of the potential. 
The method relies on a discrete variant of Agmon's comparison principle which is also proven.
Our results represent a discrete variant of the recent criteria obtained in the continuous setting by D.~Hundertmark, M.~Jex, and M.~Lange  [\emph{Forum Mathematics, Sigma} \textbf{11} (2023)].
\end{abstract}

\maketitle
\section{Introduction}
The bound states serves a crucial role in the stability of quantum systems. The special importance has a ground state as a most stable state of a given system which corresponds to the lowest eigenvalue of the Hamiltonian describing the system. We consider the Schrödinger operator
$$
H_V=-\Delta+V
$$
acting in $\ell^{2}(\Z^{d})$, where 
$$ 
(-\Delta\psi)_{n}:=\!\sum_{\substack{m\in\Z^{d} \\ |m-n|=1}}\!(\psi_{n}-\psi_{m})
$$ 
is the discrete Laplacian on the lattice $\Z^{d}$ and $V$ is a real valued potential such that $H_V$ is a self-adjoint operator on its maximal domain (more details given in Subsec.~\ref{subsec:disc_schro} below).

In this paper we focus on the situation when the ground state eigenvalue approaches the threshold of the essential spectrum. In particular, we are interested in the situation when the ground state eigenvalue is precisely at the threshold, thus becoming embedded in the essential spectrum. In general, the problem of embedded eigenvalues is of great interest. However, most of the results are focused on the eigenvalues embedded in the interior of the essential spectrum. It is known that the potential needs to satisfy specific behaviour to allow the existence of embedded eigenvalues,see e.g. \cite{vonneu93,dam05,den07,kru12,liu21,rem07}.

The conditions for the existence and absence of threshold states were studied in the continuous setting for a long time due to its importance to the time-decay of solutions of the time-dependent Schrödinger equation \cite{JeKa79,Yaf83}. Typical approach to this problem is to study the behaviour of the operator resolvent. This is much more complicated than the approach used recently in \cite{hjl23}, where the conditions were derived from subharmonic estimates on the eigenfunctions. 

The threshold states in the continuous settings were also studied in atomic systems for hydrogen anions with scaled Coulomb repulsion \cite{HofOstHofOstSim83,Hof84}, where authors shown the existence of threshold states for singlet states and absence for triplet states. This was generalized for general repulsive Coulomb interaction in \cite{bol85}, where it is shown that the long range repulsion stabilizes the threshold states. The general condition on the absence result for spherically symmetric potentials satisfying $V(x)<\frac{3}{4|x|^2}+\frac{1}{|x|^2\log|x|}$ in dimension 3 was proved in \cite{BenYar90}. The existence condition on the critical potential was derived in the form $V(x)\geq\frac{3+\epsilon}{4|x|^2}$ in \cite{GriGar07} using careful resolvent estimates. The most general condition for arbitrary dimension was done in \cite{hjl23} and its leading order terms can be written as
\begin{equation*}
\begin{split}
V(x)&\leq \frac{d(4-d)}{4|x|^{2}}+\frac{1}{|x|^{2}\log|x|}\quad\textrm{for the absence},\\
V(x)&\geq\frac{d(4-d)}{4|x|^{2}}+\frac{1+\varepsilon}{|x|^{2}\log|x|}\quad\textrm{for the existence},
\end{split}
\end{equation*}
where $d\in\N$ denotes the dimension, $\epsilon>0$, and $x$ is sufficiently large. It turns out that the properties of the zero energy eigenstates are dictated by a long range behaviour of the system also in the discrete case. 

Before we state our main results, we need to introduce a notation. The iterated logarithm $\log_{k}|n|$ is defined inductively as $\log_{0} x:=x$ for $x>0$ and $\log_{k+1}x:=\log(\log_{k}x)$ for $k\in\N_{0}$ and $x$ greater than the $k$-th tetration of $\ee$. By $\sigma(H)$, $\sigma_{\p}(H)$, and $\sigma_{\ess}(H)$ we denote the spectrum, the point spectrum, and the essential spectrum of a self-adjoint operator $H$, respectively. Our first main result is the following \textbf{absence} condition.

\begin{thm}
If there exists $s\in\N_{0}$ such that the potential $V$ of the discrete Schrödinger operator $H_{V}$ on $\Z^{d}$ fulfils
\begin{equation}\label{eq:absence}
V_{n}\leq \frac{d(4-d)}{4|n|^{2}}+\frac{1}{|n|^{2}}\sum_{j=1}^{s}\prod_{k=1}^{j}\frac{1}{\log_{k}|n|}
\end{equation}
for all $n\in\Z^{d}$ with $|n|$ sufficiently large, then $H_{V}$ does not have a zero energy ground state, i.e. $0\notin\sigma_{\p}(H_V)$ or $0\neq\inf\sigma(H_V)$.
\end{thm}

To give the complementary existence condition, we require one additional condition on the potential. We say that an operator is critical at 0 whenever an arbitrary compact negative perturbation of the operator creates discrete (negative) eigenvalues below the threshold of the essential spectrum; see Definition~\ref{def:crit} below for the exact definition. The \textbf{existence} condition formulated for critical operators in the next theorem is our second main result.

\begin{thm}
Let $\inf\sigma(H_{V})=\inf\sigma_{\ess}(H_{V})=0$ and $H_{V}$ be critical at $0$. If there exists $s\in\N_{0}$ and $\varepsilon>0$ such that 
\begin{equation}\label{eq:existence}
V_{n}\geq\frac{d(4-d)}{4|n|^{2}}+\frac{1}{|n|^{2}}\sum_{j=1}^{s}\prod_{k=1}^{j}\frac{1}{\log_{k}|n|}+\frac{\varepsilon}{|n|^{2}}\prod_{k=1}^{s}\frac{1}{\log_{k}|n|}
\end{equation}
for all $n\in\Z^{d}$ with $|n|$ sufficiently large, then $0\in\sigma_{\p}(H_{V})$.
\end{thm}

An inspection of the expressions in~\eqref{eq:absence} and \eqref{eq:existence} shows that the conditions are sharp due to the fact that $s$ can be chosen arbitrary large and $\epsilon>0$ arbitrary small. It is worth noting that the existence condition in principle works also for higher eigenvalues at the critical coupling. It is important that even though the potential estimates \eqref{eq:absence} and \eqref{eq:existence} are discrete analogies to the ones in the continuous setting \cite{hjl23}, their proofs are more involved. In fact, conditions~\eqref{eq:absence} and~\eqref{eq:existence} follow from more general Theorems~\ref{thm:absence_cond} and~\ref{thm:existence_cond} as by no means obvious  particular cases. Since the discrete Laplacian is a bounded operator, $H_V$ can have both the lower as well as the upper threshold of the essential spectrum. We prove analogous conditions to those in \eqref{eq:absence} and \eqref{eq:existence} for upper threshold eigenstates in Theorems \ref{thm:absence_cond_right} and \ref{thm:existence_cond_right} below.

From the expressions \eqref{eq:absence} and \eqref{eq:existence} we see that the critical decay of the potential is of order of order $1/|n|^{2}$ for $|n|$ large. The inverse square decay is present also in other situations as a borderline case. One is Hardy potentials, where the inverse square decay exhibit asymptotically optimal Hardy weights on $\Z^{d}$ for $d\geq3$ or $d=1$ with the Dirichlet condition at the origin. The case $d=2$ is special and reminiscent to the case $d=4$ in~\eqref{eq:absence} and~\eqref{eq:existence} for an extra logarithmic term is present in the respective Hardy weight; see~\cite{kap-lap_16} for details. Optimal discrete Hardy weights, which can be viewed as potentials $V$ which, when subtracted from the Laplacian, the corresponding Schrödinger operator becomes critical, are not completely analogical to their well-known continuous counterparts. Nevertheless, they still exhibit the inverse square decay for $d=1$ with the Dirichlet condition at the origin and $d\geq3$, see~\cite[Thms.~7.2 and~7.3]{kel-pin-pog_cmp18}.

Another example of the borderline case concerns the discrete spectrum of discrete Schrö\-din\-ger operators that is known to be finite if $V_{n}\gtrsim-|n|^{-2-\epsilon}$ for $\epsilon>0$; see \cite{naiman1959set,damanik2003bound,luef2004finiteness} for more details. Let us remark that these results, which are well-known in the continuous case, seem harder to prove in the discrete setting.

In the continuous setting the spectral phase transition was observed for $d=4$, see \cite{hjl23}. It manifests in the different behaviour of virtual levels for small dimensions and large dimensions, e.g., absence of zero energy resonances for short range potentials in dimension $d>4$, see \cite{JeKa79,Jen80,Jen84}. Our results reveal the same behaviour in the discrete setting. It follows from the dependence of the sign of the leading term in \eqref{eq:absence} and \eqref{eq:existence} on the dimension $d$. The short range potentials satisfy condition \eqref{eq:existence} for $d>4$ because the leading order term becomes negative.
 
Proofs of our main results rely on subharmonic comparison estimates in the spirit of Agmon deduced in Theorem \ref{thm:comp} below. Agmon's comparison theorem~\cite[Thm~2.7]{agm_85} proved itself to be a very useful tool in the continuous setting. We were not able to find its discrete variant in the literature and we are convinced that it is a result of independent interest in analysis of properties of eigenvectors in the discrete setting, too. Using the standard definition of subsolutions and supersolutions given precisely in Definition \ref{def:sub_sup} below, the discrete Agmon's comparison principle proves existence of a constant $C>0$ such that
$$
 u_{n}\leq Cw_{n}
$$
for all $n$ in a set, where $w_n$ is a strictly positive supersolution of the Schrödinger equation and $u_n$ a subsolution of the same Schrödinger equation satisfying an additional mild summability condition, see~\eqref{eq:agmon_assum} below. The summability condition is fulfilled when $u_n$ decays sufficiently fast as $|n|\to\infty$, see Remark~\ref{rem:subsol_decay} below. In our proofs of the absence and existence conditions, one of the functions $u$ and $w$ is always an eigenvector of a Schrödinger operator while the second is chosen as a suitable comparison function.

\subsection{Organization of the paper}

Preliminary results are deduces in four subsections of Section~\ref{sec:prelim}. Basic definitions of lattice Schrödinger operators are recalled in Subsection~\ref{subsec:disc_schro}. General ergodic theory is used to demonstrate simplicity and positivity of the ground state of a lattice Schrödinger operator in Subsection~\ref{subsec:sim_pos}. The discrete variant of Agmon's comparison principle is proven in Subsection~\ref{subsec:agmon}. Two essential inequalities are derived in Subsection~\ref{subsec:expan_ineq}.

Main results are proven in Section~\ref{sec:main}.  Conditions for the absence of the zero ground state energy are deduced in Subsection~\ref{subsec:absence}, while complementary conditions guaranteeing the existence of the ground state are given in Subsection~\ref{subsec:existence}. Similar conditions for the opposite edge point of the essential spectrum are also formulated in Subsections~\ref{subsec:absence} and~\ref{subsec:existence}. Finally, an example of a lattice Schrödinger operator demonstrating a transition between the existence and non-existence of the ground state is presented in last Subsection~\ref{subsec:example}.

\subsection{Notation}
For readers convenience, we summarize a notation used throughout this paper. As it is common, $\Z$ denotes the set of integers; $\N_{0}$ and $\N$ are the sets of non-negative and positive integers, respectively.
\begin{itemize}
\item The Euclidean norm of $n\in\Z^{d}$ is denoted by 
\[
|n|:=\sqrt{\sum_{j=1}^{d}|n_{j}|^{2}}
\]
\item The Hilbert space of square-summable functions on $\Z^{d}$ is denoted by 
\[
\ell^{2}(\Z^{d}):=\{\psi:\Z^{d}\to\C \mid \|\psi\|<\infty\},
\]
where $\|\cdot\|$ is the $\ell^{2}$-norm induced by the inner product
\[
 \langle\phi,\psi\rangle:=\sum_{n\in\Z^{d}}\overline{\phi_{n}}\psi_{n}.
\]
We use the notation $\psi_{n}:=\psi(n)$ for values of a function $\psi:\Z^{d}\to\C$.
\item For $\Omega\subset\Z^{d}$, $C_{c}(\Omega)$ denotes the space of functions $\psi:\Z^{d}\to\C$ compactly supported in $\Omega$.
\item $e_{n}$ for $n\in\Z^{d}$ and $\delta_{j}$ for $j=1,\dots,d$ denote vectors of standard bases of spaces $\ell^{2}(\Z^{d})$ and $\C^{d}$, respectively.
\item The positive part of a function $\psi:\Z^{d}\to\C$ is denoted by $\psi_{+}:=\max(0,\psi)$.
\item By $\sigma(H)$, $\sigma_{\p}(H)$, and $\sigma_{\ess}(H)$ we denote the spectrum, the point spectrum, and the essential spectrum of a self-adjoint operator $H$, respectively.
\item For $N>0$, we define
\[
 \Z_{\geq N}:=\{n\in\Z^{d} \mid |n|\geq N\}.
\]
\end{itemize}

\section{Preliminaries}\label{sec:prelim}

\subsection{Discrete Schrödinger operators}\label{subsec:disc_schro}
Recall that the \emph{discrete Laplacian} on the $d$-dimensional lattice $\Z^{d}$ is defined by the formula
\[
 (-\Delta\psi)_{n}:=\!\sum_{\substack{m\in\Z^{d} \\ |m-n|=1}}\!(\psi_{n}-\psi_{m})=2d\psi_{n}-\!\sum_{\substack{m\in\Z^{d} \\ |m-n|=1}}\psi_{m}, \quad n\in\Z^{d},
\]
for any $\psi:\Z^{d}\to\C$. When regarded as an operator on $\ell^{2}(\Z^{d})$, $-\Delta$ is bounded, self-adjoint, and diagonalized by the $d$-dimensional discrete Fourier transform
\[
 F:\ell^{2}(\Z^{d})\to L^{2}([-\pi,\pi]^{d}):\psi\mapsto\frac{1}{(2\pi)^{d/2}}\sum_{n\in\Z^{d}}e^{-\ii n\cdot\xi}\,\psi_{n}.
\]
Concretely, for any $f\in L^{2}([-\pi,\pi]^{d})$, one has the dispersion relation 
\[
 F(-\Delta)F^{-1}f(\xi)=h(\xi)f(\xi),
\]
where
\[
 h(\xi):=2\sum_{j=1}^{d}(1-\cos\xi_{j}).
\]
Consequently, the spectrum of $-\Delta$ is absolutely continuous and $\sigma(-\Delta)=[0,4d]$.

Given $V:\Z^{d}\to\R$, we define an operator by the multiplication of $V$, denoted again by the letter  $V$ with some abuse of {the} notation, by the formula 
\[
 (V\psi)_{n}:=V_{n}\psi_{n}, \quad n\in\Z^{d},
\]
on the maximal domain
\[
 \Dom V:=\left\{\psi\in\ell^{2}(\Z^{d}) \mid V\psi\in\ell^{2}(\Z^{d})\right\}.
\]
As it is standard, both the function and the operator $V$ are referred to as the \emph{potential}.
Then the \emph{discrete Schrödinger operator} on the lattice $\Z^{d}$ is the operator sum $H_{V}:=-\Delta+V$.
Since $-\Delta$ is bounded $\Dom H_{V} = \Dom V$ and $H_{V}$ is self-adjoint. 

\subsection{Simplicity and positivity of the ground state of $H_{V}$}\label{subsec:sim_pos}

It is well known that under certain assumptions, if a continuous Schrödinger operator has an eigenvalue as the lowest spectral point, this eigenvalue is simple and the corresponding eigenvector can be chosen strictly positive, see~\cite[Sec.~XIII.12]{reed-simon-vol4}. We will use the theory of~\cite[Sec.~XIII.12]{reed-simon-vol4} to show that the same holds true in the discrete setting, too. 

First, we recall a~terminology. A vector $\psi\in\ell^{2}(\Z^{d})$ is called \emph{positive} if $\psi\neq0$ and $\psi_{n}\geq0$ for all $n\in\Z^{d}$. A vector $\psi\in\ell^{2}(\Z^{d})$ is called \emph{strictly positive} if $\psi_{n}>0$ for all $n\in\Z^{d}$. A~bounded operator $A$ acting on $\ell^{2}(\Z^{d})$ is called \emph{positivity preserving} or \emph{positivity improving} if $A\psi$ is positive whenever $\psi$ is positive or $A\psi$ is strictly positive whenever $\psi$ is positive, respectively. 
Lastly, a~bounded operator $A$ is called \emph{ergodic} if $A$ is positivity preserving and for any $\phi,\psi\in\ell^{2}(\Z^{d})$ both positive there exists $k\in\N$ for which $\langle \phi,A^{k}\psi\rangle\neq0$.

A combination of results from~\cite[Sec.~XIII.12]{reed-simon-vol4} particularly yields the following useful proposition.

\begin{prop}\label{prop:reed-simon}
Let $H$ and $H_{0}$ be self-adjoint operators bounded from below on a Hilbert space and $E:=\inf\sigma(H)$ be an eigenvalue of $H$. Suppose that there exists a sequence of operators  of multiplication by bounded functions $W_{N}$ so that $H_{0}+W_{N}\to H$ and $H-W_{N}\to H_{0}$ in the strong resolvent sense as $N\to\infty$. Suppose, in addition, that $H_{0}+W_{N}$ and $H-W_{N}$ are uniformly bounded from below. If the semigroup $\exp(-tH_{0})$ is positivity improving for all $t>0$, then $E$ is a simple eigenvalue of~$H$ and the corresponding eigenvector can be chosen strictly positive.
\end{prop}

\begin{proof}
 If $\exp(-tH_{0})$ is positivity improving, it is ergodic by definition. The proof of the implication (b)$\Rightarrow$(c) in~\cite[Thm.~XIII.43]{reed-simon-vol4} shows that the set of bounded multiplication operators and $\exp(-tH_{0})$ act irreducibly (here, we do not need the assumption that $\|\exp(-tH_{0})\|$ is~an eigenvalue). Then \cite[Thm.~XIII.45]{reed-simon-vol4} tells us that also the set of bounded diagonal operators and $\exp(-tH)$ act irreducibly and $\exp(-tH)$ is positivity preserving for all $t>0$. It means, by~\cite[Thm.~XIII.43]{reed-simon-vol4}, that $\exp(-tH)$ is ergodic for all $t>0$, and~\cite[Thm.~XIII.44]{reed-simon-vol4} implies the statement.
\end{proof}

\begin{thm}\label{thm:positive}
 Let $H_{V}$ be the discrete Schrödinger operator and $E:=\inf\sigma(H_{V})>-\infty$ be an eigenvalue of $H_{V}$. Then $E$ is simple and the corresponding eigenvector can be chosen strictly positive.
\end{thm}

\begin{proof}
 For $N\in\N_{0}$, we denote by $P_{N}$ the orthogonal projection onto $\spn\{e_{n} \mid |n|\leq N\}$, where
 $\{e_{n} \mid n\in\Z^{d}\}$ is the standard basis of $\ell^{2}(\Z^{d})$, i.e. $(e_{n})_{m}=1$, if $m=n$, and $(e_{n})_{m}=0$, if $m\neq n$. Define $V_{N}:=VP_{N}=P_{N}V$. Clearly, $V_{N}$ is bounded for all $N\in\N_{0}$. Notice that $H_{V}$ is bounded from below if and only if the potential $V$ is a bounded function from below. Hence both $V_{N}$ and $V-V_{N}$ are uniformly bounded from below by the lower bound of $V$. Next, it is easy to see that $V_{N}\to V$ in the strong resolvent sense.

We intend to apply Proposition~\ref{prop:reed-simon} with $H:=H_{V}$, $H_{0}:=-\Delta-2d-1$, and $W_{N}:=V_{N}+2d+1$. If we verify that $\exp(-tH_{0})$ is positivity improving for all $t>0$, then all assumptions of Proposition~\ref{prop:reed-simon} are fulfilled and the claim follows.
To this end, we use the fact that if a bounded operator $A$ on $\ell^{2}(\Z^{d})$ satisfies
\[
\langle e_{n},Ae_{m}\rangle>0,
\] 
for all $m,n\in\Z^{d}$, then $A$ is positivity improving. We will show that this is the case for $\exp(-tH_{0})$.

Using the definition of $-\Delta$, we can express the action of $\Delta+2d+1$ to a function $\psi$ as
\[
 \left((\Delta+2d+1)\psi\right)_{n}=\psi_{n}+\!\sum_{\substack{s\in\Z^{d} \\ |s-n|=1}}\psi_{s}=\!\sum_{\substack{s\in\Z^{d} \\ |s-n|\leq1}}\psi_{s}.
\]
Then one readily checks by induction in $k\in\N$ that, for all $\psi\geq0$ and $n\in\Z^{d}$, one has
\[
 \left[(\Delta+2d+1)^{k}\psi\right]_{n}\geq \!\sum_{\substack{s\in\Z^{d} \\ |s-n|\leq k}}\psi_{s}.
\]
By the boundedness of $\Delta+2d+1$, we have
\[
 \exp(-tH_{0})=\sum_{k=0}^{\infty}\frac{t^{k}}{k!}(\Delta+2d+1)^{k}, \quad t>0,
\]
where the series converges in the operator norm. Hence for all $t>0$, $\psi\geq0$, and $n\in\Z^{d}$, we get the inequality
\[
 \left[\exp(-tH_{0})\psi\right]_{n}\geq \sum_{k=0}^{\infty}\frac{t^{k}}{k!}\!\sum_{\substack{s\in\Z^{d} \\ |s-n|\leq k}}\psi_{s}.
\]
Consequently, for any  $m,n\in\Z^{d}$ and $t>0$, we find that
\[
 \langle e_{n},\exp(-tH_{0})e_{m}\rangle\geq  \sum_{k=0}^{\infty}\frac{t^{k}}{k!}\!\sum_{\substack{s\in\Z^{d} \\ |s-n|\leq k}}\delta_{m,s}\geq\frac{t^{\ell}}{\ell!}>0,
\]
where $\ell\in\N_{0}$ is large enough so that $|m-n|\leq\ell$. The proof is complete.
\end{proof}

\subsection{Discrete Agmon's comparison principle}\label{subsec:agmon}
We deduce a discrete variant of Agmon's comparison principle given in~\cite[Thm.~2.7]{agm_85}. Although here it is derived as an essential tool for proofs of our main results, it is definitely a claim of independent interest, too.
 
First, we define discrete analogues to terms \emph{subsolution} and \emph{supersolution} of the equation $(H_{V}-\lambda)\psi=0$ in a subset of $\Z^{d}$, where $\lambda\in\R$. 

\begin{defn}\label{def:sub_sup}
 Let $\Omega\subset\Z^{d}$ and $\lambda\in\R$. Functions $u,w:\Z^{d}\to\R$ are called \emph{subsolution} and \emph{supersolution} of the equation $(H_{V}-\lambda)\psi=0$ in $\Omega$, if
 \[
   \left[(H_{V}-\lambda)u\right]_{n}\leq0 \quad\mbox{ and }\quad \left[(H_{V}-\lambda)w\right]_{n}\geq0,
 \]
 for all $n\in\Omega$, respectively.
\end{defn}

\begin{rem}
 To compare the above definition with its continuous traditional form one needs to realize that 
 $u$ is a subsolution of $(H_{V}-\lambda)\psi=0$ in $\Omega$ if and only if 
 \[
   \langle\phi,(H_{V}-\lambda)u\rangle\leq0,
 \] 
 for all \emph{non-negative} $\phi\in C_{c}(\Omega)$, where $C_{c}(\Omega)$ is the space of compactly supported functions in $\Omega$. A similar claim with the opposite inequality holds true for a supersolution~$w$.
\end{rem}

Further we will need two auxiliary observations. First notice that 
\begin{equation}
-\Delta=\sum_{j=1}^{d}D_{j}^{*}D_{j}
\label{eq:delta_parc_dif}
\end{equation}
where $D_{j}$ is the first order partial difference operator defined by
\[
\left( D_{j}\psi\right)_{n}:=\psi_{n}-\psi_{n-\delta_{j}}
\]
and $D_{j}^{*}$ its adjoint acting as 
\[
\left( D_{j}^{*}\psi\right)_{n}:=\psi_{n}-\psi_{n+\delta_{j}}
\]
for all $n\in\Z^{d}$ and $\psi:\Z^{d}\to\C$, where $\delta_{j}$ denotes the $j$-th vector of the standard basis of $\C^{d}$. 
The second observation is formulated in the following lemma. 

\begin{lem}\label{lem:v_subsol_v_+_subsol}
 If $u$ is a subsolution of $(H_{V}-\lambda)\psi=0$ in $\Omega\subset\Z^{d}$, then $u_{+}:=\max(0,u)$ is also a subsolution of $(H_{V}-\lambda)\psi=0$ in $\Omega$.
\end{lem}

\begin{proof}
Let $n\in\Omega$ is fixed. If $u_{n}\leq0$, then $(u_{+})_{n}=0$ and therefore the inequality $\left[(H_{V}-\lambda)u_{+}\right]_{n}\leq0$ is equivalent to the inequality 
 \[
  -\!\sum_{\substack{m\in\Z^{d} \\ |m-n|=1}}(u_{+})_{m}\leq0,
 \]
 which is true since $u_{+}\geq0$.
 
 If $u_{n}>0$, then $u_{n}=(u_{+})_{n}$. Taking also into account that $u\leq u_{+}$ and the assumption, we estimate $\left[(H_{V}-\lambda)u_{+}\right]_{n}$ by
 \[
 (2d+V_{n}-\lambda)(u_{+})_{n}-\!\!\sum_{\substack{m\in\Z^{d} \\ |m-n|=1}}(u_{+})_{m}\leq  (2d+V_{n}-\lambda)u_{n}-\!\!\sum_{\substack{m\in\Z^{d} \\ |m-n|=1}}u_{m}=\left[(H_{V}-\lambda)u\right]_{n}
\leq0,
 \]
 which concludes the proof.
\end{proof}

Next, we prove a discrete variant of Agmon's comparison theorem. Recall the notation
$\Z_{\geq N}^{d}:=\{n\in\Z^{d} \mid |n|\geq N\}$ for $N>0$.

\begin{thm}\label{thm:comp}
 Let $N\in\N$, $\lambda\in\R$, and $w$ be a strictly positive supersolution of $(H_{V}-\lambda)\psi=0$ in $\Z_{\geq N}^{d}$. Suppose further that $u$ is a subsolution of $(H_{V}-\lambda)\psi=0$ in $\Z^{d}_{\geq N}$ satisfying
\begin{equation}
\liminf_{M\to\infty}\frac{1}{M^{2}}\sum_{M\leq|n|\leq\alpha M}\,\sum_{j=1}^{d}|u_{n}u_{n-\delta_{j}}|=0
\label{eq:agmon_assum}
\end{equation}
for some $\alpha>1$. Then, for all $n\in\Z_{\geq N}^{d}$, one has
\begin{equation}
  u_{n}\leq Cw_{n},
\label{eq:u_leq_Cw}
\end{equation}
where $C$ is any positive constant such that 
\begin{equation}
C \geq \max\left\{\frac{u_{n}}{w_{n}} \;\,\Big|\;\, N-1\leq|n|<N+1 \right\}.
\label{eq:const_C}
\end{equation}

\end{thm}

\begin{rem}
Notice that the constant $C>0$ satisfying~\eqref{eq:const_C} always exists because 
the set of indices $n\in\Z^{d}$ with $N-1\leq|n|<N+1$ is finite and $w_{n}>0$ for all $n\in\Z^{d}$.
\end{rem}

\begin{proof}[Proof of Theorem~\ref{thm:comp}]
Pick any $C>0$ satisfying~\eqref{eq:const_C}. Notice that then inequality~\eqref{eq:u_leq_Cw} holds true for all $n\in\Z^{d}$ with $N-1\leq|n|<N+1$.
Define the auxiliary function 
\begin{equation}
 t:=(u-Cw)_{+}.
\label{eq:def_t}
\end{equation}
Then $t_{n}=0$ for $n\in\Z^{d}$ with $N-1\leq|n|<N+1$. We show that then $t$ must 
vanish identically in $\Z_{\geq N}^{d}$, which completes the proof. The proof proceeds in several steps. Throughout the proof summation indices are elements of $\Z^{d}$ restricted further by displayed inequalities.

\emph{Step~0: A summation by parts identity.} Let $\phi,\psi\in C_{c}(\Z^{d})$. We show that, if $\phi_{n}=0$ for all $n\in\Z^{d}$ such that $N-1\leq |n|< N+1$, then the identity
\begin{equation}
\sum_{|n|\geq N}\phi_{n}\left(D_{j}^{*}\psi\right)_{n}=\sum_{|n|\geq N}(D_{j}\phi)_{n}\psi_{n}
\label{eq:aux_id_inproof}
\end{equation}
holds for all $j\in\{1,\dots,d\}$. Indeed, to verify~\eqref{eq:aux_id_inproof}, we expand the left-hand side getting
\begin{equation}
\sum_{|n|\geq N}\phi_{n}\left(D_{j}^{*}\psi\right)_{n}=\sum_{|n|\geq N}\phi_{n}\psi_{n}-\sum_{|n|\geq N}\phi_{n}\psi_{n+\delta_{j}}=\sum_{|n|\geq N}\phi_{n}\psi_{n}-\sum_{|n-\delta_{j}|\geq N}\phi_{n-\delta_{j}}\psi_{n}.
\label{eq:proof_st1_inproof}
\end{equation}
Notice that if a multi-index $n\in\Z^{d}$ satisfies $|n-\delta_{j}|\geq N$ and $|n|<N$, then $N\leq |n-\delta_{j}|<N+1$ and $\phi_{n-\delta_{j}}=0$ by the assumption. Similarly, if $|n-\delta_{j}|< N$ and $|n|\geq N$, then $N-1\leq |n-\delta_{j}|<N$ and so again $\phi_{n-\delta_{j}}=0$ by the assumption. It means that the last sum in~\eqref{eq:proof_st1_inproof} remains unchanged when the restriction $|n-\delta_{j}|\geq N$ is replaced by $|n|\geq N$. Thus we arrive at the equality 
\[
\sum_{|n|\geq N}\phi_{n}\left(D_{j}^{*}\psi\right)_{n}=\sum_{|n|\geq N}\phi_{n}\psi_{n}-\sum_{|n|\geq N}\phi_{n-\delta_{j}}\psi_{n}
\]
which amounts to~\eqref{eq:aux_id_inproof}.

\emph{Step 1: An inequality from the subsolution.} We establish the identity  
\begin{equation}
\sum_{|n|\geq N}\xi_{n}^{2}t_{n}(-\Delta t)_{n}=\sum_{|n|\geq N}\sum_{j=1}^{d}\left[D_{j}(\xi t)\right]_{n}^{2}-\sum_{|n|\geq N}\sum_{j=1}^{d}t_{n}t_{n-\delta_{j}}\left(D_{j}\xi\right)_{n}^{2},
\label{eq:aux_id2_inproof}
\end{equation}
which holds true for any $t:\Z^{d}\to\R$ and $\xi\in C_{c}(\Z^{d})$ such that $\xi_{n}t_{n}=0$ whenever $N-1\leq |n|< N+1$. By applying~\eqref{eq:delta_parc_dif} and formula~\eqref{eq:aux_id_inproof} with $\phi=\xi^{2}t$ and $\psi=D_{j}t$, we find
\[
\sum_{|n|\geq N}\xi_{n}^{2}t_{n}(-\Delta t)_{n}=\sum_{j=1}^{d}\sum_{|n|\geq N}\xi_{n}^{2}t_{n}(D_{j}^{*}D_{j}t)_{n}=\sum_{j=1}^{d}\sum_{|n|\geq N}[D_{j}(\xi^{2}t)]_{n}(D_{j}t)_{n}.
\]
Next, by expanding the right-hand side and using the elementary identity
\[
 (\xi_{n}^{2}t_{n}-\xi_{n-\delta_{j}}^{2}t_{n-\delta_{j}})(t_{n}-t_{n-\delta_{j}})=(\xi_{n}t_{n}-\xi_{n-\delta_{j}}t_{n-\delta_{j}})^{2}-t_{n}t_{n-\delta_{j}}(\xi_{n}-\xi_{n-\delta_{j}})^2,
\]
we arrive at~\eqref{eq:aux_id2_inproof}.

Now suppose additionally that $t$ is a positive subsolution of $(H_{V}-\lambda)\psi=0$ in $\Z_{\geq N}^{d}$. It follows that 
\[
 0\geq \sum_{|n|\geq N}\xi_{n}^{2}t_{n}[(H_{V}-\lambda)t]_{n}=\sum_{|n|\geq N}\xi_{n}^{2}t_{n}(-\Delta t)_{n}+\sum_{|n|\geq N}(V_{n}-\lambda)\xi_{n}^{2}t_{n}^{2}.
\]
An application of identity~\eqref{eq:aux_id2_inproof} yields the inequality 
\begin{equation}
\sum_{|n|\geq N}\sum_{j=1}^{d}\left[D_{j}(\xi t)\right]_{n}^{2}+\sum_{|n|\geq N}(V_{n}-\lambda)\xi_{n}^{2}t_{n}^{2}\leq\sum_{|n|\geq N}\sum_{j=1}^{d}t_{n}t_{n-\delta_{j}}\left(D_{j}\xi\right)_{n}^{2}
\label{eq:ineq_subsol_inproof}
\end{equation}
for any $\xi\in C_{c}(\Z^{d})$ and a positive subsolution $t$ of $(H_{V}-\lambda)\psi=0$ in $\Z_{\geq N}^{d}$ satisfying $t_{n}=0$ if $N-1\leq |n|< N+1$ (to guarantee that $t_{n}\xi_{n}=0$).

\emph{Step~2: An inequality from the supersolution.} Let $w:\Z^{d}\to\R$ and $\psi\in C_{c}(\Z^{d})$ is such that $\varphi_{n}=0$ if $N-1\leq |n|< N+1$. Then identity~\eqref{eq:aux_id2_inproof} with $t$ replaced by $w$ and $\xi$ replaced by $\varphi$ reads
\begin{equation}
\sum_{|n|\geq N}w_{n}{\varphi}_{n}^{2}(-\Delta w)_{n}=\sum_{|n|\geq N}\sum_{j=1}^{d}\left[D_{j}(w{\varphi})\right]_{n}^{2}-\sum_{|n|\geq N}\sum_{j=1}^{d}w_{n}w_{n-\delta_{j}}\left(D_{j}{\varphi}\right)_{n}^{2}.
\label{eq:aux_id3_inproof}
\end{equation}
Suppose additionally that $w$ is a positive supersolution of $(H_{V}-\lambda)\psi=0$ in $\Z_{\geq N}^{d}$. Then 
\[
 0\leq \sum_{|n|\geq N}w_{n}{\varphi}_{n}^{2}[(H_{V}-\lambda)w]_{n}=\sum_{|n|\geq N}w_{n}{\varphi}_{n}^{2}(-\Delta w)_{n}+\sum_{|n|\geq N}(V_{n}-\lambda)w_{n}^{2}{\varphi}_{n}^{2}.
\]
An application of identity~\eqref{eq:aux_id3_inproof} yields the inequality 
\begin{equation}
\sum_{|n|\geq N}\sum_{j=1}^{d}\left[D_{j}(w{\varphi})\right]_{n}^{2}+\sum_{|n|\geq N}(V_{n}-\lambda)w_{n}^{2}{\varphi}_{n}^{2}\geq\sum_{|n|\geq N}\sum_{j=1}^{d}w_{n}w_{n-\delta_{j}}\left(D_{j}{\varphi}\right)_{n}^{2}
\label{eq:ineq_supersol_inproof}
\end{equation}
for any positive supersolution $w$ of $(H_{V}-\lambda)\psi=0$ in $\Z_{\geq N}^{d}$ and ${\varphi}\in C_{c}(\Z^{d})$ satisfying ${\varphi}_{n}=0$ if $N-1\leq |n|< N+1$.

\emph{Step~3: A combined inequality.} Assuming that $w$ is a \emph{strictly} positive supersolution of $(H_{V}-\lambda)\psi=0$ in $\Z_{\geq N}^{d}$ we may plug $\varphi=\xi t/w$ into~\eqref{eq:ineq_supersol_inproof}. Then its left-hand side coincides with the left-hand side of~\eqref{eq:ineq_subsol_inproof} and we deduce the inequality
\begin{equation}
\sum_{|n|\geq N}\sum_{j=1}^{d}w_{n}w_{n-\delta_{j}}\!\left[D_{j}\!\left(\frac{\xi t}{w}\right)\right]_{n}^{2}\leq \sum_{|n|\geq N}\sum_{j=1}^{d}t_{n}t_{n-\delta_{j}}\left(D_{j}\xi\right)_{n}^{2}
\label{eq:ineq_combined_inproof}
\end{equation}
for any $\xi\in C_{c}(\Z^{d})$, $t$ is a positive subsolution of $(H_{V}-\lambda)\psi=0$ in $\Z_{\geq N}^{d}$ satisfying $t_{n}=0$ if $N-1\leq |n|< N+1$, and $w$ a strictly positive supersolution $(H_{V}-\lambda)\psi=0$ in $\Z_{\geq N}^{d}$.

\emph{Step~4: Proof of the statement.}
Let the assumptions of the statement are fulfilled and define $t$ by~\eqref{eq:def_t}.
Then, by Lemma~\ref{lem:v_subsol_v_+_subsol}, $t$ is a subsolution of $(H_{V}-\lambda)\psi=0$ in $\Z_{\geq N}^{d}$ and $t_{n}=0$ for all $n\in\Z^{d}$ such that $N-1\leq |n|<N+1$.

For $M>N$ sufficiently large and $\alpha>1$, we substitute $\xi_{n}=g^{(M)}(|n|)$ into~\eqref{eq:ineq_combined_inproof} with the piece-wise linear function
\[
g^{(M)}(r):=\begin{cases}
 1 & \quad\mbox{ if }\hskip54pt r\leq M+1,\\
 \frac{\alpha M-1-r}{(\alpha-1)M-2} & \quad\mbox{ if }\; \quad M+1<r<\alpha M-1,\\
 0 & \quad\mbox{ if }\hskip6pt \alpha M-1\leq r.
\end{cases}
\]
With this choice, one finds that 
\[
 \left(D_{j}g^{(M)}\right)_{n}^{2}\leq\frac{1}{[(\alpha-1)M-2]^{2}}\,\chi_{\{M\leq|n|\leq\alpha M\}}(n)\leq\frac{{A}}{M^{2}}\,\chi_{\{M\leq|n|\leq\alpha M\}}(n)
\]
for a sufficiently large constant ${A}>0$, where $\chi$ stands for the indicator function. Taking also into account that $t\leq u_{+}\leq|u|$, we infer from~\eqref{eq:ineq_combined_inproof} that
\[
\sum_{|n|\geq N}\sum_{j=1}^{d}w_{n}w_{n-\delta_{j}}\!\left[D_{j}\!\left(\frac{g^{(M)}t}{w}\right)\right]_{n}^{2}\leq \frac{{A}}{M^{2}}\sum_{M\leq|n|\leq\alpha M}\,\sum_{j=1}^{d}|u_{n}u_{n-\delta_{j}}|
\]
for all $M$ sufficiently large. Applying Fatou's lemma and assumption~\eqref{eq:agmon_assum}, we get
\[
\sum_{|n|\geq N}\sum_{j=1}^{d}w_{n}w_{n-\delta_{j}}\!\left[D_{j}\!\left(\frac{t}{w}\right)\right]_{n}^{2}\leq \liminf_{M\to\infty}\frac{{A}}{M^{2}}\sum_{M\leq|n|\leq\alpha M}\,\sum_{j=1}^{d}|u_{n}u_{n-\delta_{j}}|=0.
\]
Since $w$ is strictly positive all terms of the sum on the left-hand side have to vanish. It implies that
\[
 \left[D_{j}\left(\frac{t}{w}\right)\right]_{n}=\frac{t_{n}}{w_{n}}-\frac{t_{n-\delta_{j}}}{w_{n-\delta_{j}}}=0
\]
for all $n\in\Z_{\geq N}^{d}$ and $j\in\{1,\dots,d\}$. Recalling the assumption $t_{n}=0$ in $N-1\leq |n|<N+1$, we conclude that $t_{n}=0$ in all of $\Z_{\geq N}^{d}$ and the proof is completed.
\end{proof}

\begin{rem}
Following exactly the definitions, the strictly positive supersolution $w$ of $(H_{V}-\lambda)\psi=0$ in $\Z^{d}_{\geq N}$ is a function $w:\Z^{d}\to(0,\infty)$ such that the inequality $[(H_{V}-\lambda))w]_{n}\geq0$ holds for every index $n\in\Z_{\geq N}^{d}$. However, the proof of Theorem~\ref{thm:comp} takes only the values of $w_{n}$ with $|n|\geq N-1$ into account. For $|n|<N-1$, values of $w_{n}$ can be taken arbitrary positive numbers. 
Therefore in proofs of Theorems~\ref{thm:absence_cond} and~\ref{thm:existence_cond} below, where Theorem~\ref{thm:comp} is applied with $N$ large, we need not care when $w$ (and similarly $u$) are chosen as functions not defined on a ball of finite radius.
\end{rem}

\begin{rem}\label{rem:subsol_decay}
If the subsolution $u$ in Theorem~\ref{thm:comp} satisfies $u_{n}=O(|n|^{-\gamma})$ for $|n|\to\infty$,
with $\gamma>(d-2)/2$, then~\eqref{eq:agmon_assum} holds true. Indeed, since the number of lattice points in $\R^{d}$ inside a ball of radius $R$ equals $O(R^{d})$, as $R\to\infty$, we deduce that 
\[
\sum_{M\leq|n|\leq\alpha M}1=O(M^{d}), \quad\mbox{ as } M\to\infty,
\]
for any $\alpha>1$. Hence, with a sufficiently large constant $C>0$, we have
\[
\frac{1}{M^{2}}\sum_{M\leq|n|\leq\alpha M}\,\sum_{j=1}^{d}|u_{n}u_{n-\delta_{j}}|\leq\frac{C}{M^{2+2\gamma}}\sum_{M\leq|n|\leq\alpha M}1=O(M^{d-2\gamma-2}), \quad\mbox{ as } M\to\infty.
\]
\end{rem}

\subsection{Expansions and inequalities}\label{subsec:expan_ineq}

We prove two, in a sense complementary, inequalities which will be essential in the forthcoming derivation of the absence and existence conditions for potentials. To this end, we recall that by $\log_{k}$ we denote the composition of $k$ natural logarithms with the convention that $\log_{0}$ stands for the identity, i.e.,
\[
\log_{0} x=x \quad \mbox{ and } \quad \log_{k}x=\log(\log_{k-1}x) \;\mbox{ for } k\in\N.
\]
Clearly, $\log_{k}x>0$ if $x>\ee_{k-1}$, where 
\[
\ee_{-1}=0 \quad \mbox{ and } \quad \ee_{k}=\exp \ee_{k-1} \;\mbox{ for } k\in\N_{0}.
\]
Further, for $s\in\N_{0}$ and $\varepsilon\geq0$, we define positive functions 
\begin{equation}
 b_{n}^{s}(\varepsilon):=|n|^{-d/2}\left(\prod_{i=1}^{s}\log_{i}^{-1/2}|n|\right)\log_{s}^{-\varepsilon/4}|n|,
\label{eq:def_b_s_eps}
\end{equation}
where $n\in\Z^{d}$ with $|n|>\ee_{s-1}$. For $\varepsilon=0$, we briefly write
\begin{equation}
 b_{n}^{s}:=b_{n}^{s}(0)=|n|^{-d/2}\prod_{i=1}^{s}\log_{i}^{-1/2}|n|.
\label{eq:def_b_s}
\end{equation}
The two key inequalities will be direct consequences of the following asymptotic expansions.

\begin{prop}
 Let $s\in\N$ and $\varepsilon>0$. For $|n|\to\infty$, we have
 \begin{equation}
 \frac{(\Delta b^{s})_{n}}{b_{n}^{s}}=\frac{d(4-d)}{4|n|^{2}}+\frac{1}{|n|^{2}}\sum_{j=1}^{s}\prod_{k=1}^{j}\frac{1}{\log_{k}|n|}+\frac{3}{4|n|^{2}\log^{2}|n|}+O\left(\frac{1}{|n|^{2}\log^{2}|n|\log_{2}|n|}\right)
 \label{eq:b_s_expan}
 \end{equation}
and
\begin{equation}
 \frac{(\Delta b^{s}(\varepsilon))_{n}}{b_{n}^{s}(\varepsilon)}=\frac{d(4-d)}{4|n|^{2}}+\frac{1}{|n|^{2}}\sum_{j=1}^{s}\prod_{k=1}^{j}\frac{1}{\log_{k}|n|}+\frac{\varepsilon}{2|n|^{2}}\prod_{k=1}^{s}\frac{1}{\log_{k}|n|}+O_{\varepsilon}\left(\frac{1}{|n|^{2}\log^{2}|n|}\right),
 \label{eq:b_s_eps_expan}
\end{equation}
where the index $\varepsilon$ in the second Landau symbol indicates the dependence of the reminder on~$\varepsilon$.
\end{prop}

\begin{proof}
\emph{Step~0: A~notation.}
It turns out to be advantageous to start with expansions in terms of the quantity
\[
 N_{j}:=\frac{2n_{j}+1}{|n|^{2}},
\]
where $j\in\{1,\dots,d\}$. Notice that $N_{j}=O(1/|n|)$ as $|n|\to\infty$. Further, for $k\in\N_{0}$, we define
\[
 x_{k}:=\frac{\log_{k}|n+\delta_{j}|}{\log_{k}|n|},
\]
where the dependence on $n\in\Z^{d}$ and $j\in\{1,\dots,d\}$ is suppressed in the notation. Whenever needed, the norm of $n\in\Z^{d}$ is assumed to be sufficiently large so the iterated logarithms are well defined, no division by zero occurs, etc. Finally, we will abbreviate 
\[
{\ell}_{k}:=\log_{k}|n|.
\]

\emph{Step~1: Asymptotic expansion of $x_{k}$.}
Notice that, for $k\in\N$, we have the recurrence
\begin{equation}
 x_{k}=1+\frac{\log x_{k-1}}{{\ell}_{k}}.
\label{eq:x_k_recur}
\end{equation}
As the first step, we deduce auxiliary expansions of $x_{k}$ for $|n|$ large. Specifically, we prove that, for $k\in\N$ and $j\in\{1,\dots,d\}$, we have
\begin{equation}
 x_{k}=1+\frac{N_{j}}{2{\ell}_{1}\dots{\ell}_{k}}-\frac{N_{j}^{2}}{4{\ell}_{1}\dots{\ell}_{k}}\left(1+\frac{1}{2}\sum_{r=1}^{k-1}\frac{1}{{\ell}_{1}\dots{\ell}_{r}}\right)+O\left(N_{j}^{3}\right),
\label{eq:x_k_expan}
\end{equation}
as $|n|\to\infty$.

The proof of~\eqref{eq:x_k_expan} proceeds by induction in $k$. If $k=1$, then 
\[
 x_{1}=1+\frac{\log(1+N_{j})}{2{\ell}_{1}}.
\]
Using the elementary expansion
\begin{equation}
 \log(1+X)=X-\frac{X^{2}}{2}+O(X^{3}), \quad X\to0,
\label{eq:log_taylor}
\end{equation}
we find that 
\[
x_{1}=1+\frac{1}{2{\ell}_{1}}\left(N_{j}-\frac{N_{j}^{2}}{2}+O\left(N_{j}^{3}\right)\right)=1+\frac{N_{j}}{2{\ell}_{1}}-\frac{N_{j}^{2}}{4{\ell}_{1}}+O\left(N_{j}^{3}\right), \quad |n|\to\infty,
\]
as claimed.

Suppose $k\geq2$ is fixed and formula~\eqref{eq:x_k_expan} holds for $x_{k-1}$. Then recurrence~\eqref{eq:x_k_recur} and the induction hypothesis imply
\[
 x_{k}=1+\frac{1}{{\ell}_{k}}\log\left(1+\frac{N_{j}}{2{\ell}_{1}\dots{\ell}_{k-1}}-\frac{N_{j}^{2}}{4{\ell}_{1}\dots{\ell}_{k-1}}\left(1+\frac{1}{2}\sum_{r=1}^{k-2}\frac{1}{{\ell}_{1}\dots{\ell}_{r}}\right)+O\left(N_{j}^{3}\right)\right),
\]
as $|n|\to\infty$. Using again~\eqref{eq:log_taylor}, we find 
\begin{align*}
 x_{k}&=1+\frac{1}{{\ell}_{k}}\left[\frac{N_{j}}{2{\ell}_{1}\dots{\ell}_{k-1}}-\frac{N_{j}^{2}}{4{\ell}_{1}\dots{\ell}_{k-1}}\left(1+\frac{1}{2}\sum_{r=1}^{k-2}\frac{1}{{\ell}_{1}\dots{\ell}_{r}}\right)-\frac{N_{j}^{2}}{8{\ell}_{1}^{2}\dots{\ell}_{k-1}^{2}}\right] 
 +O\left(N_{j}^{3}\right) \\
 &=1+\frac{N_{j}}{2{\ell}_{1}\dots{\ell}_{k}}-\frac{N_{j}^{2}}{4{\ell}_{1}\dots{\ell}_{k}}\left(1+\frac{1}{2}\sum_{r=1}^{k-1}\frac{1}{{\ell}_{1}\dots{\ell}_{r}}\right)
 +O\left(N_{j}^{3}\right),
\end{align*}
for $|n|\to\infty$, which concludes the proof of~\eqref{eq:x_k_expan}.

\emph{Step~2: Proof of expansion~\eqref{eq:b_s_expan}.}
Notice that the expression from the left side of~\eqref{eq:b_s_expan} can be expressed in terms of $x_{k}$ as follows:
\begin{equation}
\frac{(\Delta b^{s})_{n}}{b_{n}^{s}}=-2d+\sum_{j=1}^{d}\frac{b_{n+\delta_{j}}^{s}}{b_{n}^{s}}+\frac{b_{n-\delta_{j}}^{s}}{b_{n}^{s}}
=-2d+\sum_{j=1}^{d}\left(x_{0}^{-d/2}\prod_{k=1}^{s}x_{k}^{-1/2}+\mbox{s.c.}\right).
\label{eq:delta_b_aux_inproof}
\end{equation}
Above and hereafter we use the abbreviation $\mbox{s.c.}$ for the same term as the one displayed, where each occurrence of $n_{j}$ is replaced by $-n_{j}$, for brevity.

With the aid of \eqref{eq:x_k_expan} and the expansion
\[
 (1+X)^{-1/2}=1-\frac{X}{2}+\frac{3X^2}{8}+ O(X^3), \quad X\to0,
\]
we deduce, for $k\in\N$, that 
\[
 x_{k}^{-1/2}=1-\alpha_{k}N_{j}+\beta_{k}N_{j}^{2}+O\left(N_{j}^{3}\right), \quad |n|\to\infty,
\]
where 
\begin{equation}
\alpha_{k}=\alpha_{k}({\ell}_{1},\dots,{\ell}_{k}):=\frac{1}{4{\ell}_{1}\dots{\ell}_{k}}
\label{eq:def_alpha_k}
\end{equation}
and
\begin{equation}
\beta_{k}=\beta_{k}({\ell}_{1},\dots,{\ell}_{k}):=\frac{1}{8{\ell}_{1}\dots{\ell}_{k}}\left(1+\frac{1}{2}\sum_{r=1}^{k-1}\frac{1}{{\ell}_{1}\dots{\ell}_{r}}+\frac{3}{4{\ell}_{1}\dots{\ell}_{k}}\right).
\label{eq:def_beta_k}
\end{equation}
Taking the product over $k=1,\dots,s$ yields the formula
\begin{equation}
\prod_{k=1}^{s}x_{k}^{-1/2}=1-\left(\sum_{k=1}^{s}\alpha_{k}\right)N_{j}+\left(\sum_{k=1}^{s}\beta_{k}+\sum_{k=1}^{s}\sum_{l=k+1}^{s}\alpha_{k}\alpha_{l}\right)N_{j}^{2}+O\left(N_{j}^{3}\right)
\label{eq:prod_x_k_part_expan}
\end{equation}
for $|n|\to\infty$.

Similarly, since 
\begin{equation}
(1+X)^{-d/4}=1-\frac{dX}{4}+\frac{d(d+4)X^2}{32}+ O(X^3), \quad X\to0,
\label{eq:elem_expand_d_inproof}
\end{equation}
we readily get
\[
 x_{0}^{-d/2}=\left(1+N_{j}\right)^{-d/4}=1-\frac{dN_{j}}{4}+\frac{d(d+4)N_{j}^2}{32}+ O\left(N_{j}^3\right), \quad |n|\to\infty.
\]
Multiplying the above formula with~\eqref{eq:prod_x_k_part_expan} yields
\begin{align}
x_{0}^{-d/2}\prod_{k=1}^{s}x_{k}^{-1/2}=1&-\left(\frac{d}{4}+\sum_{k=1}^{s}\alpha_{k}\right)N_{j}\nonumber\\
&+\left(\frac{d(d+4)}{32}+\frac{d}{4}\sum_{k=1}^{s}\alpha_{k}+\sum_{k=1}^{s}\beta_{k}+\sum_{k=1}^{s}\sum_{l=k+1}^{s}\alpha_{k}\alpha_{l}\right)N_{j}^{2}+O\left(N_{j}^{3}\right) \label{eq:x_0_prod_x_k_part_expan}
\end{align}
for $|n|\to\infty$. Taking also into account that 
\begin{equation}
 N_{j}+\mbox{s.c.}=\frac{2}{|n|^{2}}
 \quad\mbox{ and }\quad
 N_{j}^{2}+\mbox{s.c.}=\frac{8n_{j}^{2}+2}{|n|^{4}},
\label{eq:N_j_N_j_squared_sum}
\end{equation}
we deduce, for $|n|\to\infty$, the expansion 
\begin{align*}
\sum_{j=1}^{d}\left(x_{0}^{-d/2}\prod_{k=1}^{s}x_{k}^{-1/2}+\mbox{s.c.}\right)&=2d-\left(\frac{d}{4}+\sum_{k=1}^{s}\alpha_{k}\right)\frac{2d}{|n|^{2}}\\
&\hskip-38pt+\left(\frac{d(d+4)}{32}+\frac{d}{4}\sum_{k=1}^{s}\alpha_{k}+\sum_{k=1}^{s}\beta_{k}+\sum_{k=1}^{s}\sum_{l=k+1}^{s}\alpha_{k}\alpha_{l}\right)\frac{8}{|n|^{2}}+O\left(\frac{1}{|n|^{3}}\right),
\end{align*}
where we have also used that $N_{j}=O(1/|n|)$ as $|n|\to\infty$.

Recalling~\eqref{eq:delta_b_aux_inproof}, a slight simplification of the last formula provides us with the expansion
\[
\frac{(\Delta b^{s})_{n}}{b_{n}^{s}}=\left(\frac{d(4-d)}{4}+8\sum_{k=1}^{s}\beta_{k}+8\sum_{k=1}^{s}\sum_{l=k+1}^{s}\alpha_{k}\alpha_{l}\right)\frac{1}{|n|^{2}}+O\left(\frac{1}{|n|^{3}}\right)
\]
for $|n|\to\infty$. Finally, substituting for $\alpha_{k}$ and $\beta_{k}$ from~\eqref{eq:def_alpha_k} and \eqref{eq:def_beta_k}, one finds that the coefficient in front of $1/|n|^{2}$ equals
\[
\frac{d(4-d)}{4}+\sum_{k=1}^{s}\frac{1}{{\ell}_{1}\dots{\ell}_{k}}+\frac{3}{4{\ell}_{1}^{2}}+O\left(\frac{1}{{\ell}_{1}^{2}{\ell}_{2}}\right)
\]
for $|n|\to\infty$. The proof of formula~\eqref{eq:b_s_expan} follows.

\emph{Step~3: Proof of expansion~\eqref{eq:b_s_eps_expan}.}
An analogous expression to~\eqref{eq:delta_b_aux_inproof} for the left hand side of~\eqref{eq:b_s_eps_expan} reads
\[
 \frac{\left(\Delta b^{s}(\varepsilon)\right)_{n}}{b_{n}^{s}(\varepsilon)}
=-2d+\sum_{j=1}^{d}\left[x_{0}^{-d/2}\left(\prod_{k=1}^{s}x_{k}^{-1/2}\right)x_{s}^{-\varepsilon/4}+\mbox{c.s.}\right].
\]
When compared to~\eqref{eq:delta_b_aux_inproof}, one sees that the proof proceeds similarly as in the previous step taking only into account the presence of the additional term $x_{s}^{-\varepsilon/4}$. 

Using~\eqref{eq:elem_expand_d_inproof} with $d$ replaced by $\varepsilon$ together with~\eqref{eq:x_k_expan}, we get 
\[
x_{s}^{-\varepsilon/4}=1-\frac{\varepsilon N_{j}}{8{\ell}_{1}\dots{\ell}_{s}}+\frac{\varepsilon N_{j}^{2}}{16{\ell}_{1}\dots{\ell}_{s}}\left(1+\frac{1}{2}\sum_{r=1}^{s-1}\frac{1}{{\ell}_{1}\dots{\ell}_{r}}+\frac{\varepsilon+4}{8{\ell}_{1}\dots{\ell}_{s}}\right)+O_{\varepsilon}\left(N_{j}^{3}\right),
\]
which, when rewritten in terms of~\eqref{eq:def_alpha_k} and~\eqref{eq:def_beta_k}, reads
\[
x_{s}^{-\varepsilon/4}=1-\frac{\varepsilon\alpha_{s}}{2}\,N_{j}+\frac{4\varepsilon\beta_{s}+\varepsilon(\varepsilon-2)\alpha_{s}^{2}}{8}\,N_{j}^{2}+O_{\varepsilon}\left(N_{j}^{3}\right)
\]
for $|n|\to\infty$. Multiplying the above expansion with~\eqref{eq:x_0_prod_x_k_part_expan} yields
\[
x_{0}^{-d/2}\left(\prod_{k=1}^{s}x_{k}^{-1/2}\right)x_{s}^{-\varepsilon/4}=1-A_{s}N_{j}+B_{s}N_{j}^{2}+O_{\varepsilon}\left(N_{j}^{3}\right)
\]
for $|n|\to\infty$, where
\[
 A_{s}:=\frac{d}{4}+\sum_{k=1}^{s}\alpha_{k}+\frac{\varepsilon\alpha_{s}}{2}
\]
and
\[
B_{s}:=\frac{d(d+4)}{32}+\frac{d}{4}\sum_{k=1}^{s}\alpha_{k}+\sum_{k=1}^{s}\beta_{k}+\sum_{k=1}^{s}\sum_{l=k+1}^{s}\!\alpha_{k}\alpha_{l}+\frac{\varepsilon\alpha_{s}}{2}\left(\frac{d}{4}+\sum_{k=1}^{s}\alpha_{k}+\frac{(\varepsilon-2)\alpha_{s}}{4}\right)+\frac{\varepsilon\beta_{s}}{2}.
\]

By summing up with respect to $j=1,\dots,d$, recalling~\eqref{eq:N_j_N_j_squared_sum}, and using that $N_{j}=O(1/|n|)$ for $|n|\to\infty$, we arrive at the expansion
\[
\sum_{j=1}^{d}\left[x_{0}^{-d/2}\left(\prod_{k=1}^{s}x_{k}^{-1/2}\right)x_{s}^{-\varepsilon/4}+\mbox{c.s.}\right]=2d-\frac{2dA_{s}-8B_{s}}{|n|^{2}}+O_{\varepsilon}\left(\frac{1}{|n|^{3}}\right)
\]
Noticing further that 
\[
B_{s}=\frac{d(d+4)}{32}+\left(\frac{d}{4}+\frac{1}{2}\right)\sum_{k=1}^{s}\alpha_{k}+\left(\frac{d}{4}+\frac{1}{2}\right)\frac{\varepsilon\alpha_{s}}{2}+O_{\varepsilon}\!\left(\frac{1}{{\ell}_{1}^{2}}\right),
\]
we compute
\[
 2dA_{s}-8B_{s}=\frac{d(d-4)}{4}-4\sum_{k=1}^{s}\alpha_{k}-2\varepsilon\alpha_{s}+O_{\varepsilon}\!\left(\frac{1}{{\ell}_{1}^{2}}\right)
\] 
for $|n|\to\infty$. In total, we obtain the expansion 
\[
\frac{\left(\Delta b^{s}(\varepsilon)\right)_{n}}{b_{n}^{s}(\varepsilon)}
=\left(\frac{d(4-d)}{4}+4\sum_{k=1}^{s}\alpha_{k}+2\varepsilon\alpha_{s}\right)\frac{1}{|n|^{2}}+O_{\varepsilon}\!\left(\frac{1}{|n|^{2}{\ell}_{1}^{2}}\right)
\]
for $|n|\to\infty$. Finally, substituting from~\eqref{eq:def_alpha_k}, we infer~\eqref{eq:b_s_eps_expan}.
\end{proof}

As an immediate corollary, we obtain from expansions~\eqref{eq:b_s_expan} and~\eqref{eq:b_s_eps_expan} the following inequalities.

\begin{cor}
 Let $s\in\N$ and $\varepsilon>0$. For all $n\in\Z^{d}$ with $|n|$ sufficiently large, we have inequalities
 \begin{equation}\label{eq:b_s_ineq}
 \frac{(\Delta b^{s})_{n}}{b_{n}^{s}}\geq\frac{d(4-d)}{4|n|^{2}}+\frac{1}{|n|^{2}}\sum_{j=1}^{s}\prod_{k=1}^{j}\frac{1}{\log_{k}|n|}
 \end{equation}
 and
 \begin{equation}\label{eq:b_s_eps_ineq}
  \frac{(\Delta b^{s}(\varepsilon))_{n}}{b_{n}^{s}(\varepsilon)}\leq\frac{d(4-d)}{4|n|^{2}}+\frac{1}{|n|^{2}}\sum_{j=1}^{s}\prod_{k=1}^{j}\frac{1}{\log_{k}|n|}+\frac{\varepsilon}{|n|^{2}}\prod_{k=1}^{s}\frac{1}{\log_{k}|n|},
 \end{equation}
 where sequences $b^{s}(\varepsilon)$ and $b^{s}$ are defined in~\eqref{eq:def_b_s_eps} and~\eqref{eq:def_b_s}.
\end{cor}

\begin{rem}
 If $s=0$, inequality~\eqref{eq:b_s_ineq} does not hold.
\end{rem}

\section{The absence and existence conditions}\label{sec:main}

The main goal of this section is to derive conditions for potential $V$ of the discrete Schrödinger operator $H_{V}$ implying either absence or existence of the zero energy ground state of $H_{V}$. The absence condition is a sufficient condition restricting only the entries of the potential from above for indices with sufficiently large Euclidean norm. The existence condition requires a complementary restriction on the entries of the potential from below and, in addition, a criticality of the Schrödinger operator in question. Analogous results concerning the right-most spectral point will be also derived.

The strategy of proofs relies on the discrete Agmon's comparison principle (Theorem~\ref{thm:comp}). As comparison sequences, $b^{s}(\varepsilon)$ and $b^{s}\equiv b^{s}(0)$ defined by~\eqref{eq:def_b_s_eps} and~\eqref{eq:def_b_s} will be used. It is important to notice that, for any $\varepsilon\geq0$, $s\in\N_{0}$, and $N\geq\ee_{s-1}$, $b^{s}_{n}(\varepsilon)>0$ for all $n\in\Z^{d}_{\geq N}$, and 
\[
b^{s}(\varepsilon)\in\ell^{2}\!\left(\Z^{d}_{\geq N}\right) \;\mbox{ if } \varepsilon>0, \mbox{ but } b^{s}\notin\ell^{2}\!\left(\Z^{d}_{\geq N}\right).
\]

\subsection{The absence condition}\label{subsec:absence}

Recall that the discrete Schrödinger operator $H_{V}$ is said to have a zero energy ground state if and only if $0=\inf\sigma(H_{V})\in\sigma_{\p}(H_{V})$.

\begin{thm}\label{thm:absence_cond}
If there exists $s\in\N_{0}$ such that potential $V$ of the discrete Schrödinger operator $H_{V}$ on $\Z^{d}$ fulfills
\begin{equation}
 V_{n}\leq -2d+\sum_{j=1}^{d}\left(\frac{|n|}{|n+\delta_{j}|}\right)^{d/2}\prod_{i=1}^{s}\sqrt{\frac{\log_{i}|n|}{\log_{i}|n+\delta_{j}|}}+\left(\frac{|n|}{|n-\delta_{j}|}\right)^{d/2}\prod_{i=1}^{s}\sqrt{\frac{\log_{i}|n|}{\log_{i}|n-\delta_{j}|}}
\label{eq:V_cond_gen_absence}
\end{equation}
for all $n\in\Z^{d}$ with $|n|$ sufficiently large, then $H_{V}$ does not have a zero energy ground state, i.e. $0\notin\sigma_{\p}(H_V)$ or $0\neq\inf\sigma(H_V)$. In particular, it is the case if
\begin{equation}
V_{n}\leq \frac{d(4-d)}{4|n|^{2}}+\frac{1}{|n|^{2}}\sum_{j=1}^{s}\prod_{k=1}^{j}\frac{1}{\log_{k}|n|}
\label{eq:V_cond_ineq_absece}
\end{equation}
for all $n\in\Z^{d}$ with $|n|$ sufficiently large.
\end{thm}

\begin{proof}
Suppose $V$ satisfies~\eqref{eq:V_cond_gen_absence} for some $s\in\N_{0}$ and all $n\in\Z^{d}_{\geq N}$, where $N$ is a sufficiently large positive integer. Notice that~\eqref{eq:V_cond_gen_absence} is equivalent to the inequality 
\[
 V_{n}\leq\frac{(\Delta b^{s})_{n}}{b_{n}^{s}},
\]
where the sequence $b^{s}$ is defined by~\eqref{eq:def_b_s}. Therefore 
\[
\left(H_{V}b^{s}\right)_{n}=(-\Delta b^{s})_{n}+V_{n}b_{n}^{s}\leq0
\]
for all $n\in\Z_{\geq N}^{d}$,
and so $b^{s}$ is a subsolution of the equation $H_{V}\psi=0$ in $\Z^{d}_{\geq N}$.

For a contradiction, suppose $0=\inf\sigma(H_{V})\in\sigma_{\p}(H_{V})$. Then the eigenvector $\phi\in\ell^{2}(\Z^{d})$ of $H_{V}$ corresponding to the eigenvalue $0$ can be chosen strictly positive by Theorem~\ref{thm:positive}. We apply Theorem~\ref{thm:comp} to vectors $w:=\phi$ and $u:=b^{s}$. {Using Remark~\ref{rem:subsol_decay}, one sees that} all assumptions of Theorem~\ref{thm:comp} are satisfied, indeed. Theorem~\ref{thm:comp} implies the existence of a~constant $C>0$ such that 
\[
 b_{n}^{s}\leq C\phi_{n}
\] 
for all $n\in\Z_{\geq N}^{d}$. It follows, however, that $b^{s}\in\ell^{2}(\Z_{\geq N}^{d})$, which is a contradiction.

The second claim follows readily from inequality~\eqref{eq:b_s_ineq} for $s\in\N$. The case $s=0$ also follows as the right-hand side of~\eqref{eq:V_cond_ineq_absece} is an increasing function of $s$.
\end{proof}

Recall that the spectrum of $H_{0}=-\Delta$ equals $[0,4d]$. Since the spectrum of the unperturbed operator is a compact interval and conditions of Theorem~\ref{thm:absence_cond} concern its left end-point $0$, it is relevant to ask whether there is a similar condition also for the right end-point $4d$. Of course, this question has no continuous counterpart since the classical Laplacian regarded as an operator acting in $L^{2}(\R^{d})$ is not bounded from above. In the discrete setting, such a~complementary condition exists and is analogous to~\eqref{eq:V_cond_gen_absence}.

\begin{thm}\label{thm:absence_cond_right}
If there exists $s\in\N_{0}$ such that potential $V$ of the discrete Schrödinger operator $H_{V}$ on $\Z^{d}$ fulfills
\begin{equation}
 V_{n}\geq 2d-\sum_{j=1}^{d}\left(\frac{|n|}{|n+\delta_{j}|}\right)^{d/2}\prod_{i=1}^{s}\sqrt{\frac{\log_{i}|n|}{\log_{i}|n+\delta_{j}|}}+\left(\frac{|n|}{|n-\delta_{j}|}\right)^{d/2}\prod_{i=1}^{s}\sqrt{\frac{\log_{i}|n|}{\log_{i}|n-\delta_{j}|}}
\label{eq:V_cond_gen_absence_right}
\end{equation}
for all $n\in\Z^{d}$ with $|n|$ sufficiently large, then $4d\notin\sigma_{\p}(H_{V})$ or $4d\neq\sup\sigma(H_{V})$. In particular, it is the case if
\[
V_{n}\geq-\frac{d(4-d)}{4|n|^{2}}-\frac{1}{|n|^{2}}\sum_{j=1}^{s}\prod_{k=1}^{j}\frac{1}{\log_{k}|n|}
\]
for all $n\in\Z^{d}$ with $|n|$ sufficiently large.
\end{thm}

\begin{proof}
Consider the unitary involution $U$ on $\ell^{2}(\Z^{d})$ defined by the equation
\[
 (U\psi)_{n}:=(-1)^{n_{1}+\dots+n_{d}}\,\psi_{n}
\]
for any $n\in\Z^{d}$. Then, by using the respective definitions, one readily verifies that 
\[
U(-\Delta)U=4d+\Delta,
\]
which implies the relation
\begin{equation}
H_{-V}=4d-UH_{V}U
\label{eq:minus_pot_rel}
\end{equation} 
for any potential $V$.

It follows from~\eqref{eq:minus_pot_rel} that
\[
 \inf\sigma(H_{-V})=4d+\inf-\sigma(H_{V})=4d-\sup\sigma(H_{V})
\]
and
\[
 0\in\sigma_{\p}(H_{-V}) \quad\Longleftrightarrow\quad 4d\in\sigma_{\p}(H_{V}).
\]
Now it suffices to note that, if $V$ fulfils~\eqref{eq:V_cond_gen_absence_right}, then $-V$ satisfies~\eqref{eq:V_cond_gen_absence}, and apply Theorem~\ref{thm:absence_cond} to $H_{-V}$.
\end{proof}

\begin{rem}
Spectral properties of $H_{V}$ are particularly understood if $d=1$ in which case $H_{V}$ is a Jacobi operator. We complement our conditions with  other results. When imposed jointly, the simplest form of conditions~\eqref{eq:V_cond_gen_absence} and~\eqref{eq:V_cond_gen_absence_right} for $s=0$ requires
\begin{equation}
 |V_{n}|\leq 2-\sqrt{\frac{n}{n+1}}-\sqrt{\frac{n}{n-1}}
\label{eq:V_modulus_cond}
\end{equation}
for all $n\in\Z$ of sufficiently large modulus. It follows that $V_{n}=O(1/n^{2})$, as $|n|\to\infty$. So $V$ is a trace class operator which implies that there are no embedded eigenvalues of $H_{V}$ in $(0,4)$. This is a well know discrete analogue to a result of the scattering theory about an asymptotic behavior of the Jost solutions; see for instance~\cite[Eq.~1.17]{gol_ieot21}. Hence the potential~$V$ can produce only eigenvalues in $(-\infty,0]\cup[4,\infty)$. If we restrict the class of potentials even more to non-trivial potentials satisfying 
\begin{equation}
 \sum_{n\in\Z}n^{1+\varepsilon}\,|V_{n}|<\infty \quad\mbox{ and }\quad \sum_{n\in\Z}V_{n}\leq0
\label{eq:sum_weak_coupling}
\end{equation}
for some $\varepsilon>0$, then the weak coupling analysis tells us that there is always a negative eigenvalue in the spectrum of $H_{V}$, see~\cite[Thm.~1.4]{kho-lak-alm_jmaa21} or~\cite[Thm.~A.19]{hoa-hun-ric-vug_ahp23}. Similarly, if we alter the second inequality in~\eqref{eq:sum_weak_coupling} to $\geq0$, there is always an eigenvalue of $H_{V}$ greater than $4$. Hence, in these particular cases, Theorems~\ref{thm:absence_cond} and~\ref{thm:absence_cond_right} do not imply anything new. However, condition~\eqref{eq:sum_weak_coupling} is stronger than~\eqref{eq:V_modulus_cond}. For a~potential that fulfills~\eqref{eq:V_modulus_cond} (or its logarithmic refinements) but not~\eqref{eq:sum_weak_coupling}, Theorems~\ref{thm:absence_cond} and~\ref{thm:absence_cond_right} yield new results even for the simplest $d=1$ case. If, in addition to~\eqref{eq:V_modulus_cond}, $V$ does not produce weakly coupled bound states, i.e., $H_{V}$ enjoys the spectral stability $\sigma(H_{V})=\sigma_{\ess}(H_{V})$, our results imply that the spectrum of $H_{V}$ is purely continuous and fills the interval $[0,4]$. Conditions for spectral stability of discrete Schrödinger operators on $\N$ allowing even complex-valued potentials have been studied recently in~\cite[Sec.~5]{kre-lap-sta_blms22}.
\end{rem}

\subsection{The existence condition}\label{subsec:existence}

First, we define the notion of criticality adapted to the discrete setting. As the discrete Schrödinger operators $H_{V}$ are bounded if and only if $V$ is bounded, in contrast to the continuous case, one may define two kinds of criticality of $H_{V}$ reflecting the fact that the spectrum of $H_{V}$ can have two finite end points. Below letters $V$ and $W$ are generically used for multiplication operators equipped with their maximal domains.

\begin{defn}\label{def:crit}
 Suppose $H_{V}$ is the discrete Schrödinger operator on $\Z^{d}$ which is bounded from below and denote $s_{-}:=\inf\sigma(H_{V})$. Then we call $H_{V}$ to be 
 \begin{enumerate}[i)]
 \item \emph{critical at $s_{-}$ (from below)} if and only if 
  \[
   (\forall W\geq0 \mbox{ compact})(\,\inf\sigma(H_{V}-W)=s_{-} \;\Rightarrow\; W=0),
  \]
 \item \emph{subcritical at $s_{-}$ (from below)} if and only if 
  \[
   (\exists W\geq0 \mbox{ compact})(\inf\sigma(H_{V}-W)=s_{-} \mbox{ and } W\neq0).
  \]
 \end{enumerate}
 Similarly, assuming $H_{V}$ to be bounded from above and $s_{+}:=\sup\sigma(H_{V})$, then $H_{V}$ is called
 \begin{enumerate}[i)]
 \item \emph{critical at $s_{+}$ (from above)} if and only if 
  \[
   (\forall W\geq0 \mbox{ compact})(\,\sup\sigma(H_{V}+W)=s_{+} \;\Rightarrow\; W=0),
  \]
 \item \emph{subcritical at $s_{+}$ (from above)} if and only if 
  \[
   (\exists W\geq0 \mbox{ compact})(\sup\sigma(H_{V}+W)=s_{+} \mbox{ and } W\neq0).
  \]
 \end{enumerate}
\end{defn}

\begin{rem}
Analogously to the continuous case, the discrete Laplacian $H_{0}$ on the lattice $\Z^{d}$ is critical at $0$ for $d=1,2$, which demonstrates the existence of weakly coupled bound states~\cite{kho-lak-alm_jmaa21}, and subcritical at $0$ for $d\geq3$, which follows from the existence of Hardy inequalities~\cite{roz-sol_09, kap-lap_16}. A theory on critical Schrödinger operators on lattices and more general graph structures is discussed in~\cite{kel-pin-pog_cmp18} and~\cite{kel-pin-pog_jst20}.
\end{rem}

Now we are ready to state the condition for the existence of the zero energy ground state of $H_{V}$ on $\Z^{d}$.

\begin{thm}\label{thm:existence_cond}
 Let $\inf\sigma(H_{V})=\inf\sigma_{\ess}(H_{V})=0$ and $H_{V}$ be critical at $0$. If there exist $s\in\N_{0}$ and $\varepsilon>0$ such that 
\begin{align}
V_{n}\geq-2d&+\sum_{j=1}^{d}\left(\frac{|n|}{|n+\delta_{j}|}\right)^{d/2}\left(\prod_{i=1}^{s}\sqrt{\frac{\log_{i}|n|}{\log_{i}|n+\delta_{j}|}}\right)\left(\frac{\log_{s}|n|}{\log_{s}|n+\delta_{j}|}\right)^{\varepsilon}\nonumber\\
&+\sum_{j=1}^{d}\left(\frac{|n|}{|n-\delta_{j}|}\right)^{d/2}\left(\prod_{i=1}^{s}\sqrt{\frac{\log_{i}|n|}{\log_{i}|n-\delta_{j}|}}\right)\left(\frac{\log_{s}|n|}{\log_{s}|n-\delta_{j}|}\right)^{\varepsilon}
\label{eq:V_cond_gen_existence}
\end{align}
for all $n\in\Z^{d}$ with $|n|$ sufficiently large, then $0\in\sigma_{\p}(H_{V})$. In particular, if there exist $s\in\N_{0}$ and $\varepsilon>0$ such that 
\begin{equation}
V_{n}\geq\frac{d(4-d)}{4|n|^{2}}+\frac{1}{|n|^{2}}\sum_{j=1}^{s}\prod_{k=1}^{j}\frac{1}{\log_{k}|n|}+\frac{\varepsilon}{|n|^{2}}\prod_{k=1}^{s}\frac{1}{\log_{k}|n|}
\label{eq:V_cond_ineq_existence}
\end{equation}
for all $n\in\Z^{d}$ with $|n|$ sufficiently large, then $0\in\sigma_{\p}(H_{V})$.
\end{thm}

\begin{proof}
 Let us denote by $\delta_{0}$ the Dirac delta potential which acts on $\ell^{2}(\Z^{d})$ as 
 \[ 
 \left(\delta_{0}\psi\right)_{n}=\begin{cases}
 \psi_{0} &\quad\mbox{ if } n=0,\\
 0 &\quad\mbox{ if } n\neq0.
 \end{cases}
 \]
 Suppose $H_{V}$ satisfies the assumptions. We define the auxiliary sequence of potentials
 \[
  V^{k}:=V-\frac{1}{k}\delta_{0}, \quad k\in\N. 
 \]
 Then $H_{V^k}=H_{V}-\delta_{0}/k$ and $\Dom H_{V^k}=\Dom H_{V}=\Dom V$. Since $H_{V}$ is critical at $0$, $\inf\sigma(H_{V^k})<0$ for all $k\in\N$. Moreover, $\sigma_{\ess}(H_{V^k})=\sigma_{\ess}(H_{V})\subset[0,\infty)$ because $\delta_{0}$ is a~rank one operator and hence $H_{V^k}-H_{V}$ is compact. It follows that
\[
E_{k}:=\inf\sigma\left(H_{V^k}\right)
\]
is a~discrete eigenvalue of $H_{V^k}$. By the min-max principle, $-1/k\leq E_{k}<0$, therefore $E_{k}\to 0$ as $k\to\infty$.
Let us denote by $\phi^{k}$ the \emph{normalized} eigenvector of $H_{V^k}$ corresponding to the eigenvalue $E_{k}$. By Theorem~\ref{thm:positive}, we may assume $\phi^{k}$ to be strictly positive.
 
As the unit ball in any reflexive Banach space is weakly precompact, $\{\phi^{k}\}_{k=1}^{\infty}$ contains a~weakly convergent subsequence which we again denote by $\{\phi^{k}\}_{k=1}^{\infty}$ with some abuse of the notation. Notice that the weak convergence in $\ell^{2}(\Z^{d})$ means nothing but the point-wise convergence. In the course of the proof, we will show that, under the assumption~\eqref{eq:V_cond_gen_existence}, $\{\phi^{k}\}_{k=1}^{\infty}$ converges even strongly, i.e. in the norm of $\ell^{2}(\Z^{d})$. To this end, it suffices to verify that $\|\phi\|=1$, where $\phi$ is the weak limit of $\{\phi^{k}\}_{k=1}^{\infty}$.

Condition~\eqref{eq:V_cond_gen_existence} is equivalent to the inequality 
\[
 V_{n}\geq\frac{(\Delta b^{s}(\delta))_{n}}{b^{s}(\delta)_{n}},
\]
where $\delta:=4\varepsilon$ and sequence $b^{s}(\delta)$ is defined by~\eqref{eq:def_b_s_eps}.
In other words, $\left[H_{V}b^{s}(\delta)\right]_{n}\geq0$ for all $n\in\Z_{\geq N}^{d}$, where $N$ is a positive integer. Taking also into account that $\left[\delta_{0}b^{s}(\delta)\right]_{n}=0$ for $n\neq0$, we get
\[
 \left[\left(H_{V^k}-E_{k}\right)b^{s}(\delta)\right]_{n}=\left[H_{V}b^{s}(\delta)\right]_{n}-E_{k}b^{s}_{n}(\delta)\geq-E_{k}b^{s}_{n}(\delta)>0,
\]
for all $n\in\Z_{\geq N}^{d}$. Hence $b^{s}(\delta)$ is a strictly positive supersolution of the equation $(H_{V^{k}}-E_{k})\psi=0$ in $\Z_{\geq N}^{d}$. 

Fix $N\in\N$ such that assumption~\eqref{eq:V_cond_gen_existence} holds for all $n\in\Z^{d}_{\geq N}$.
For any $k\in\N$, we may apply Theorem~\ref{thm:comp} to $w:=b^{s}(\delta)$ and $u:=\phi^k$, which implies
\[
 \phi_{n}^{k}\leq C_{k} b^{s}_{n}(\delta).
\]
for all $n\in\Z_{\geq N}^{d}$, with the constant $C_{k}>0$ chosen as
\[
 C_{k}:=\max_{N-1\leq|n|<N+1}\frac{\phi_{n}^{k}}{b_{n}^{s}(\delta)}.
\]
By the normalization and positivity of all $\phi^{k}$, we have $0\leq \phi_{n}^{k}\leq1$ for all $k\in\N$ and $n\in\Z^{d}$. Then 
\[
 C_{k}\leq C:=\max_{N-1\leq|n|<N+1}\frac{1}{b_{n}^{s}(\delta)}
\]
for all $k\in\N$ and we obtain the estimate 
\[
 \phi_{n}^{k}\leq C b^{s}_{n}(\delta)
\]
valid for all $n\in\Z_{\geq N}^{d}$ with the $k$-independent constant $C$. Since $b^{s}(\delta)\in\ell^{2}(\Z_{\geq M}^{d})$ it follows
\[
 \|\phi\|^{2}=\sum_{n\in\Z^{d}}|\phi_{n}|^{2}=\lim_{k\to\infty}\sum_{n\in\Z^{d}}\left|\phi_{n}^{k}\right|^{2}=1
\]
by the Lebesgue dominated convergence. Thus, $\phi^{k}\to\phi$ in the norm of $\ell^{2}(\Z^{d})$.

To finish the proof of the first claim we show that vector $\phi$ is an eigenvector of $H_{V}$ corresponding to the eigenvalue $0$. Clearly, $0\neq\phi\in\ell^{2}(\Z^{d})$ since $\|\phi\|=1$. So we are done once we show that $\phi\in\Dom V$ and, as $k\to\infty$,
\[
\left(H_{V^{k}}-E_{k}\right)\phi^{k}\to H_{V}\phi
\]
in $\ell^{2}(\Z^{d})$. Since $E_{k}\to0$, $\delta_{0}$ is bounded, and $\phi^{k}$ are uniformly bounded, both $k^{-1}\delta_{0}\phi^{k}$ and $E_{k}\phi^{k}$ tend to $0$ in $\ell^{2}(\Z^{d})$. Further, $\Delta$ is bounded and so $\Delta\phi^{k}\to\Delta\phi$ in $\ell^{2}(\Z)$. Thus, it suffices to show that $\phi\in\Dom V$ and $V\phi^{k}\to\ V\phi$ in $\ell^{2}(\Z^{d})$. This is true if $V\phi^{k}$ converges in $\ell^{2}(\Z^{d})$ because $V$ is a closed operator on its maximal domain. Finally, the last assertion is true indeed, because
\[
 V\phi^{k}=\Delta\phi^{k}+\frac{1}{k}\delta_{0}\phi^{k}+E_{k}\phi^{k}\to\Delta\phi
\]
in $\ell^{2}(\Z^{d})$.

To verify the second claim for $s\in\N$ it suffices to note that, if $V$ satisfies~\eqref{eq:V_cond_ineq_existence}, then $V$ fulfils~\eqref{eq:V_cond_gen_existence} with $\varepsilon$ replaced by $\varepsilon/4$ by inequality~\eqref{eq:b_s_eps_ineq}. Factor $1/4$ is of course inessential and we may apply the already proven first claim. Finally, the case $s=0$ is also covered since, if condition~\eqref{eq:V_cond_ineq_existence} is true for $s=0$, then it is true also for $s=1$ and all $n\in\Z^{d}$ with $|n|$ sufficiently large.
\end{proof}

\begin{rem}
It is not obvious that potentials satisfying assumptions of Theorem~\ref{thm:existence_cond} really exist. This is demonstrated by an example in Section~\ref{subsec:example}.
\end{rem}

The variant of Theorem~\ref{thm:existence_cond} for the right spectral end-point can be proven in a similar fashion as is Theorem~\ref{thm:absence_cond_right} deduced from Theorem~\ref{thm:absence_cond}. Note that it follows from equation~\eqref{eq:minus_pot_rel} that $H_{V}$ is critical at $4d$ from above if and only if $H_{-V}$ is critical at $0$ from below.

\begin{thm}\label{thm:existence_cond_right}
 Let $\sup\sigma(H_{V})=\sup\sigma_{\ess}(H_{V})=4d$ and $H_{V}$ be critical at $4d$ from above. If there exist $s\in\N_{0}$ and $\varepsilon>0$ such that 
\begin{align*}
V_{n}\leq2d&-\sum_{j=1}^{d}\left(\frac{|n|}{|n+\delta_{j}|}\right)^{d/2}\left(\prod_{i=1}^{s}\sqrt{\frac{\log_{i}|n|}{\log_{i}|n+\delta_{j}|}}\right)\left(\frac{\log_{s}|n|}{\log_{s}|n+\delta_{j}|}\right)^{\varepsilon}\\
&-\sum_{j=1}^{d}\left(\frac{|n|}{|n-\delta_{j}|}\right)^{d/2}\left(\prod_{i=1}^{s}\sqrt{\frac{\log_{i}|n|}{\log_{i}|n-\delta_{j}|}}\right)\left(\frac{\log_{s}|n|}{\log_{s}|n-\delta_{j}|}\right)^{\varepsilon}
\end{align*}
for all $n\in\Z^{d}$ with $|n|$ sufficiently large, then $4d\in\sigma_{\p}(H_{V})$. In particular, if there exist $s\in\N_{0}$ and $\varepsilon>0$ such that 
\[
V_{n}\leq-\frac{d(4-d)}{4|n|^{2}}-\frac{1}{|n|^{2}}\sum_{j=1}^{s}\prod_{k=1}^{j}\frac{1}{\log_{k}|n|}-\frac{\varepsilon}{|n|^{2}}\prod_{k=1}^{s}\frac{1}{\log_{k}|n|}
\]
for all $n\in\Z^{d}$ with $|n|$ sufficiently large, then $4d\in\sigma_{\p}(H_{V})$.
\end{thm}

\subsection{An example}\label{subsec:example}

For a parameter $\gamma>0$, we consider potential $V=V(\gamma)$ defined by 
\begin{equation}
 V_{n}:=\frac{(\Delta a)_{n}}{a_{n}}, \quad n\in\Z^{d},
\label{eq:def_V_gamma_a}
\end{equation}
where
\begin{equation}
 a_{n}:=\begin{cases}
  |n|^{-\gamma}& \quad\mbox{ if } n\neq0,\\
  1& \quad\mbox{ if } n=0.\\
 \end{cases}
\label{eq:def_a_gamma}
\end{equation}
Explicitly, the entries of $V$ read
\[
 V_{n}=\begin{cases}
  0 &\mbox { if } n=0,\\
  1+2^{-\gamma}+(d-1)2^{1-\gamma/2}-2d &\mbox { if } n=\pm\delta_{i},\, i\in\{1,\dots,d\},\\
 -2d+{\displaystyle \sum_{j=1}^{d}\frac{|n|^{\gamma}}{|n+\delta_{j}|^{\gamma}}+\frac{|n|^{\gamma}}{|n-\delta_{j}|^{\gamma}}} &\mbox { otherwise. }\\
 \end{cases}
\]

The existence or non-existence of the zero energy ground state of the Schrödinger operator $H_{V}=-\Delta+V$ depend on the value of parameter $\gamma$, which shows the following proposition.

\begin{prop}
Let $\gamma>0$ and $H_{V}$ be the discrete Schrödinger operator with potential defined by~\eqref{eq:def_V_gamma_a} and~\eqref{eq:def_a_gamma}. Then one has:
\begin{enumerate}[i)]\setlength\itemsep{2pt}
\item $\sigma_{\ess}(H_{V})=[0,4d]$,
\item $H_{V}$ is bonded and non-negative,
\item $H_{V}$ is critical at $0$ for $\gamma>(d-1)/2$,
\item $0\in\sigma_{\p}(H_{V})$ if and only if $\gamma>d/2$.
\end{enumerate}
\end{prop}

\begin{proof}
i) 
A straightforward computation shows that
\begin{equation}
 V_{n}=\frac{\gamma(\gamma+2-d)}{|n|^{2}}+O\left(\frac{1}{|n|^{3}}\right), \quad |n|\to\infty.
\label{eq:V_n_asympt}
\end{equation}
It follows that $V$ is a compact operator and therefore 
$\sigma_{\ess}(H_{V})=\sigma_{\ess}(-\Delta)=[0,4d]$ by the Weyl criterion.

ii) 
The boundedness of $H_{V}$ follows from boundedness of $V$ which, in its turn, is obvious from~\eqref{eq:V_n_asympt}.
 
Next, we prove that $H_{V}\geq0$. Since $H_{V}$ is bounded it suffices to verify that $\langle \psi,H_{V}\psi\rangle\geq0$ for all compactly supported sequences $\psi\in\ell^{2}(\Z^{d})$. In addition, we can assume that $\psi$ is real without loss of generality. Define the auxiliary sequence $\phi$ by putting $\phi_{n}:=\psi_{n}/a_{n}$ for all $n\in\Z^{d}$, where the positive sequence $a$ is given by~\eqref{eq:def_a_gamma}. Then using~\eqref{eq:delta_parc_dif} and~\eqref{eq:def_V_gamma_a}, we find
\begin{align*}
 \langle \psi,H_{V}\psi\rangle &= \sum_{j=1}^{d}\|D_{j}(a\phi)\|^{2}+\langle a\phi,Va\phi\rangle\\ 
   &= \sum_{j=1}^{d}\sum_{n\in\Z^{d}}(a_{n}\phi_{n}-a_{n-\delta_j}\phi_{n-\delta_j})^{2}-2d\sum_{n\in\Z^{d}}a_{n}^{2}\phi_{n}^{2}+\sum_{n\in\Z^{d}}a_{n}\phi_{n}^{2}\sum_{j=1}^{d}\left(a_{n+\delta_{j}}+a_{n-\delta_{j}}\right)\\
  &=-2\sum_{j=1}^{d}\sum_{n\in\Z^{d}}a_{n}a_{n-\delta_j}\phi_{n}\phi_{n-\delta_j}+\sum_{j=1}^{d}\sum_{n\in\Z^{d}}a_{n}a_{n-\delta_j}\phi_{n-\delta_j}^{2}+\sum_{j=1}^{d}\sum_{n\in\Z^{d}}a_{n}a_{n-\delta_j}\phi_{n}^{2}\\
 &=\sum_{j=1}^{d}\sum_{n\in\Z^{d}}a_{n}a_{n-\delta_{j}}(\phi_{n}-\phi_{n-\delta_j})^{2}\geq0.
\end{align*}

iii) 
Using the above computation with $\psi^{N}:=a\phi^{N}$, where 
\[
 \phi_{n}^{N}:=\begin{cases}
  1& \quad\mbox{ if }\; n\in[-N,N]^{d},\\
  0& \quad\mbox{ if }\; n\notin[-N,N]^{d},
 \end{cases} 
\]
one finds that
\[
\langle \psi^{N},H_{V}\psi^{N}\rangle=\sum_{j=1}^{d}\left(\sum_{\substack{n\in[-N,N]^{d} \\ n_j=N+1}}a_{n}a_{n-\delta_{j}}+\sum_{\substack{n\in[-N,N]^{d} \\ n_j=-N}}a_{n}a_{n-\delta_{j}}\right)
\]
for any $N\in\N$.  Taking definition~\eqref{eq:def_a_gamma} into account, the right hand side can be further simplified getting
\begin{align*}
\langle \psi^{N},H_{V}\psi^{N}\rangle&= 2d\sum_{n_2=-N}^{N}\dots\sum_{n_d=-N}^{N}a_{(N+1,n_2,\dots,n_d)}a_{(N,n_2,\dots,n_d)}\\
&\leq 2d\sum_{n_2=-N}^{N}\dots\sum_{n_d=-N}^{N}a_{(N,n_2,\dots,n_d)}^{2}\\
&=\frac{2d}{N^{2\gamma-d+1}}\sum_{n_2=-N}^{N}\dots\sum_{n_d=-N}^{N}\frac{1}{N^{d-1}}\left[1+\left(\frac{n_2}{N}\right)^{2}+\dots+\left(\frac{n_d}{N}\right)^{2}\right]^{-\gamma}.
\end{align*}
Notice the last multi-sum, which is to be interpreted as $1$ if $d=1$, is the~Riemann sum for a~multi-variable function. Consequently, as $N\to\infty$, we have the finite limit
\[
\sum_{n_2=-N}^{N}\dots\sum_{n_d=-N}^{N}\frac{1}{N^{d-1}}\left[1+\left(\frac{n_2}{N}\right)^{2}+\dots+\left(\frac{n_d}{N}\right)^{2}\right]^{-\gamma}\to\int_{-1}^{1}\dots\int_{-1}^{1}\frac{\dd x_{2}\dots\dd x_{d}}{(1+x_{2}^{2}+\dots+x_{d}^{2})^{\gamma}}.
\]
Therefore we may conclude that there exists a constant ${C}>0$ such that 
\[
\langle \psi^{N},H_{V}\psi^{N}\rangle\leq \frac{{C}}{N^{2\gamma-d+1}}
\]
for all $N\in\N$.

Now suppose that $2\gamma>d-1$ and $W\geq0$ is a compact potential such that $H_{V}\geq W$ in the sense of quadratic forms. We show that $W=0$ which implies claim~(iii). For all $N\in\N$, we have
\[
\sum_{n\in[-N,N]^{d}}a_{n}^{2}W_{n}=\langle \psi^{N},W\psi^{N}\rangle\leq\langle \psi^{N},H_{V}\psi^{N}\rangle\leq \frac{{C}}{N^{2\gamma-d+1}}.
\]
Taking the limit $N\to\infty$, we find that
\[
\sum_{n\in\Z^{d}}a_{n}^{2}W_{n}\leq0.
\]
Since $W_{n}\geq0$ and $a_{n}>0$ for all $n\in\Z^{d}$, the last inequality implies that $W_{n}=0$ for all $n\in\Z^{d}$, i.e. $W=0$.

iv) 
Suppose $\gamma>d/2$. Then $a\in\ell^{2}(\Z^{d})$ and it follows readily from definition~\eqref{eq:def_V_gamma_a} that $H_{V}a=0$. 
Alternatively, one deduces the same result from~\eqref{eq:V_n_asympt} and Theorem~\ref{thm:existence_cond} with $s=0$.

If $0<\gamma<d/2$, then a comparison of~\eqref{eq:V_n_asympt} with the condition~\eqref{eq:V_cond_ineq_absece} of Theorem~\ref{thm:absence_cond}
for $s=0$ yields that $H_{V}$ does not have a~zero energy ground state. Recalling that $H_{V}\geq0$, it follows that $0\notin\sigma_{\p}(H_{V})$.
If $\gamma=d/2$, one proceeds similarly using the refined condition~\eqref{eq:V_cond_ineq_absece} of Theorem~\ref{thm:absence_cond} with $s=1$. 
\end{proof}

\section*{Acknowledgment}
M.~J. received financial support from the Ministry of Education, Youth and Sport of the Czech Republic under the Grants No. RVO 14000. M.~J. is grateful for financial support from ”Centre for Advanced Applied Sciences”, Registry No. CZ.02.1.01/0.0/0.0/16 019/0000778, supported by the Operational Programme Research, Development and Education, co-financed by the European Structural and Investment Funds. F.~{\v S}. acknowledges the support of the EXPRO grant No.~20-17749X of the Czech Science Foundation.

\bibliographystyle{acm}

\end{document}